\documentclass{article}
\usepackage{arxiv}
\usepackage[utf8]{inputenc} 
\usepackage[T1]{fontenc}    
\usepackage{hyperref}       
\usepackage{url}            
\usepackage{booktabs}       
\usepackage{amsfonts}       
\usepackage{nicefrac}       
\usepackage{microtype}      
\usepackage{lipsum}		
\usepackage{graphicx}
\usepackage{mathrsfs}
\usepackage{doi}
\usepackage{amsmath,amsthm}
\newtheorem{theorem}{Theorem}
\usepackage{algorithm}
\usepackage{algorithmicx}
\usepackage[noend]{algpseudocode}
\usepackage{caption,subcaption}
\allowdisplaybreaks
\algnewcommand{\LeftComment}[1]{\Statex \(\triangleright\) #1}
\usepackage{stmaryrd}
\usepackage{color,soul}
\usepackage{wrapfig}
\usepackage{cite}
\usepackage{cancel}
\hypersetup{
    colorlinks=true,
    linkcolor=blue,
    filecolor=magenta,      
    urlcolor=cyan,
}
\usepackage{courier}

\usepackage{listings}
\usepackage{xcolor}
\definecolor{codegreen}{rgb}{0,0.6,0}
\definecolor{codegray}{rgb}{0.5,0.5,0.5}
\definecolor{codepurple}{rgb}{0.58,0,0.82}
\definecolor{backcolour}{rgb}{0.95,0.95,0.92}

\lstdefinestyle{mystyle}{
  backgroundcolor=\color{backcolour},   commentstyle=\color{codegreen},
  keywordstyle=\color{magenta},
  numberstyle=\tiny\color{codegray},
  stringstyle=\color{codepurple},
  basicstyle=\ttfamily\footnotesize,
  breakatwhitespace=false,         
  breaklines=true,                 
  captionpos=b,                    
  keepspaces=true,                 
  numbers=left,                    
  numbersep=5pt,                  
  showspaces=false,                
  showstringspaces=false,
  showtabs=false,                  
  tabsize=2
}

\title{Design of convergence criterion for fixed stress split iterative scheme
 for small strain anisotropic poroelastoplasticity coupled with single phase
 flow}

\author{
  {
  Saumik Dana}\\
	University of Southern California\\
	Los Angeles, CA 90007 \\
	\texttt{sdana@usc.edu} \\
 		\And
  {
  Mary F Wheeler} \\
	Oden Institute for Computational Engineering and Sciences\\
	University of Texas at Austin\\
	Austin, TX 78712 
}

\date{}

\begin{document}
\maketitle
          \begin{abstract}
               The convergence criterion for the fixed stress split iterative scheme for single phase flow coupled with small strain anisotropic poroelastoplasticity is derived. The analysis is based on studying the equations satisfied by the difference of iterates to show that the iterative scheme is contractive. The contractivity is based on driving a term to as small a value as possible (ideally zero). This condition is rendered as the convergence criterion of the algorithm. 

            \keywords{Staggered solution algorithm \and Anisotropic
            poroelastoplasticity \and Coupled flow and mechanics \and
            Contraction map \and Convergence criterion}
\end{abstract}

\section{Introduction}\label{intro}
\begin{figure}[htb!]
\centering
\includegraphics[scale=0.4]{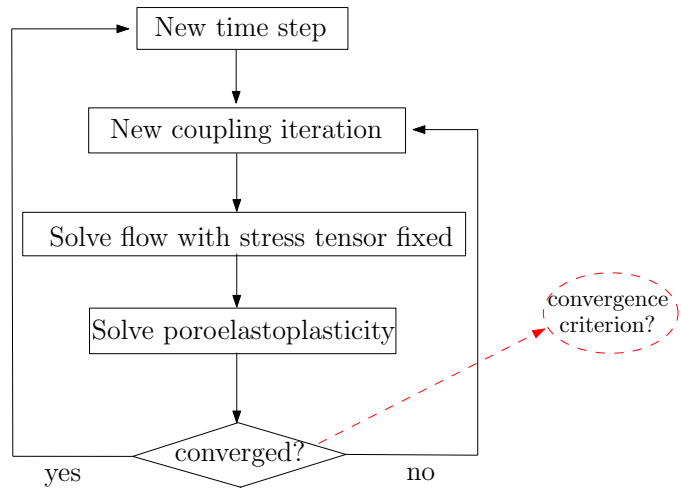}
\caption{Fixed stress split iterative scheme for anisotropic poroelastoplasticity coupled with single phase flow. Our objective is to use the framework of contraction mapping to design a convergence criterion for the algorithm.}
\label{algorithm}
\end{figure}
\par\noindent
With regard to the numerical simulation of multiphysics phenomena and the advent of modular object oriented programming frameworks, individual physical models are written in separate templates in most modern code frameworks (Deal.II \cite{bangerth}, FEniCS \cite{logg}, Feel++ \cite{prud}, FEI (Trilinos) \cite{heroux}, FreeFEM++ \cite{hecht}, libMESH \cite{kirk}, MFEM \cite{anderson}, OOFEM \cite{patzak} and SOFA \cite{faure}). While most of these frameworks are being improved upon on a regular basis, there are specific algorithms designed for specific problems that are developed in-house in research labs. It becomes imperative to collate the features of the open source frameworks with the functionalities of the in-house code developments. In that regard, staggered solution algorithms offer avenues for fast and easy collaboration. These algorithms solve individual physics individually by decoupling the coupled set of partial differential equations and then iterate back and forth in each time step. These algorithms, unlike fully coupled systems, are swift and easy to implement, but they come at the cost of lack of convergence norms. Each staggered solution algorithm needs to be carefully designed with regards to two important aspects: the decoupling constraint and the convergence criterion.   
The coupling of fluid diffusion with deformation in elasto-plastic porous
materials is important from the standpoint of variety of engineering
applications not limited to geomaterials~\cite{inria,cowin}. The
fixed stress split iterative algorithm has been studied for the coupling of
fluid diffusion with deformation in elastic porous material. In this work,
we study the fixed stress split algorithm~\cite{dana-2018,dana2019system,dana2020,dana2020efficient,dana2021,danacg,danacmame,danathesis,JAMMOUL2020,jammoul2019RSC} for coupling fluid diffusion with
deformation in elasto-plastic porous material. We use the framework of a contraction mapping to arrive at the convergence
criterion for the algorithm coupling small strain anisotropic
poroelastoplasticity with single phase flow. The reader is referred to \mbox{\cite{storvik}} and \mbox{\cite{both2017}} for recent work using the framework of contraction mapping. As shown in Figure \ref{algorithm}, the flow subproblem is solved with stress tensor fixed followed by the poromechanics subproblem in every coupling iteration at each time step. The coupling iterations are repeated until convergence and Backward Euler is employed for time marching. The analysis is motivated by the results in our previous works as follows
\begin{itemize}
\item In \cite{danacmame}, a contraction mappping for two-grid staggered algorithm led to closed form expressions for coarse scale moduli in terms of fine scale data. Here, the flow equations were solved on a fine grid and the isotropic poroelasticity equations were solved on a coarse grid. 
\item In \cite{danacg}, a contraction mapping is applied to demonstrate convergence of staggered solution algorithm for anisotropic poroelasticity coupled with single phase flow. The speciality of this algorithm was that the stress tensor is fixed during the flow solve as an extension to the case with isotropic poroelasticity in which the mean stress was fixed during the flow solve.
\end{itemize}
This paper is divided into the following sections.  Model equations with
corresponding variational formulations being considered are introduced in
Section 2. Convergence criterion is derived in Section 3. A summary of
conclusions are presented in Section 4.
\section{Model equations and variational statements}
We present the equations of the linear flow model and the poromechanics model in this section.
\subsection{Linear flow model}\label{flowmodel1}
Let the boundary $\partial \Omega=\Gamma_D^f \cup \Gamma_N^f$ where $\Gamma_D^f$ is the Dirichlet boundary and $\Gamma_N^f$ is the Neumann boundary. The equations are 
\begin{equation}
\label{floweq}
\left.\begin{array}{c}
\frac{\partial \zeta}{\partial t}+\nabla \cdot \mathbf{z}=q\qquad(\mathrm{Mass\,\,conservation})\\
\mathbf{z}=-\frac{\mathbf{K}}{\mu}(\nabla p-\rho_0 \mathbf{g})=-\boldsymbol{\kappa}(\nabla p-\rho_0 \mathbf{g})\qquad(\mathrm{Darcy's\,\,law})\\
\rho=\rho_0(1+c\,p)\\
p=g \,\, \mathrm{on}\,\,\Gamma_D^f \times (0,T],\,\,\mathbf{z}\cdot\mathbf{n}=0 \,\, \mathrm{on}\,\,\Gamma_N^f \times (0,T]\\
p(\mathbf{x},0)=0,\,\,\rho(\mathbf{x},0)=\rho_0(\mathbf{x}),\,\, \phi(\mathbf{x},0)=\phi_0(\mathbf{x})\qquad (\forall \mathbf{x}\in \Omega)
\end{array}\right.
\end{equation}
where $p:\Omega \times (0,T]\rightarrow \mathbb{R}$ is the fluid pressure, $\mathbf{z}:\Omega \times (0,T]\rightarrow \mathbb{R}^3$ is the fluid flux, $\zeta$ is the increment in fluid content\footnote{\cite{biot-energy} defines the increment in fluid content as the measure of the amount of fluid which has flowed in and out of a given element attached to the solid frame}, $q$ is the source or sink term, $\mathbf{K}$ is the uniformly symmetric positive definite absolute permeability tensor, $\mu$ is the fluid viscosity, $\phi$ is the porosity, $\boldsymbol{\kappa}=\frac{\mathbf{K}}{\mu}$ is a measure of the hydraulic conductivity of the pore fluid, $c$ is the fluid compressibility and $T>0$ is the time interval. 
\subsection{Poromechanics model}\label{poromodel1}
The important phenomenological aspects of small strain elastoplasticity are
\begin{itemize}
\item The existence of an elastic domain, i.e. a range of stresses within which the behaviour
of the material can be considered as purely elastic, without evolution of permanent
(plastic) strains. The elastic domain is delimited by the so-called yield stress. A scalar yield function $\Phi(\boldsymbol{\sigma})$ is introduced. 
The yield locus is the boundary of the elastic domain where $\Phi(\boldsymbol{\sigma})=0$ and the corresponding yield surface is defined as $\mathcal{Y}=\{\boldsymbol{\sigma}|\Phi(\boldsymbol{\sigma})=0\}$.
\item If the material is further loaded at the yield stress, then plastic yielding (or plastic flow),
i.e. evolution of plastic strains, takes place. 
\end{itemize}
Let the boundary $\partial \Omega=\Gamma_D^p \cup \Gamma_N^p$ where $\Gamma_D^p$ is the Dirichlet boundary and $\Gamma_N^p$ is the Neumann boundary. The equations are 
\begin{equation}
\label{poromecheq}
\left.\begin{array}{c}
\nabla\cdot \boldsymbol{\sigma}+\mathbf{f}=\mathbf{0}\qquad(\mathrm{Linear\,\,momentum\,\,balance})\\
\mathbf{f}=\rho \phi\mathbf{g} + \rho_r(1-\phi)\mathbf{g}\\
\boldsymbol{\epsilon}^e(\mathbf{u})=\frac{1}{2}(\nabla \mathbf{u} + (\nabla \mathbf{u})^T)\equiv \boldsymbol{\epsilon}(\mathbf{u})-\boldsymbol{\epsilon}^p(\mathbf{u})\qquad(\mathrm{small\,\,strain\,\,elastoplasticity})\\
\boldsymbol{\epsilon}^p=\gamma\frac{\partial \Phi}{\partial \boldsymbol{\sigma}}\qquad(\mathrm{plastic\,\,flow\,\,rule})\\
\boldsymbol{\sigma}=\mathbb{D}\boldsymbol{\epsilon}^e-\boldsymbol{\alpha} p\equiv \mathbb{D}(\boldsymbol{\epsilon}-\gamma\frac{\partial \Phi}{\partial \boldsymbol{\sigma}})-\boldsymbol{\alpha} p\qquad(\mathrm{constitutive\,\,law})\\
\mathbf{u}=\mathbf{0}\,\, \mathrm{on}\,\,\Gamma_D^p \times [0,T],\,\,
\boldsymbol{\sigma}^T\mathbf{n}=\mathbf{t}\,\, \mathrm{on}\,\,\Gamma_N^p \times [0,T]\\
\mathbf{u}(\mathbf{x},0)=\mathbf{0}\qquad (\forall \mathbf{x}\in \Omega)
\end{array}\right.
\end{equation}
where $\mathbf{u}:\Omega \times [0,T]\rightarrow \mathbb{R}^3$ is the solid displacement, $\rho_r$ is the rock density, $\mathbf{f}$ is the body force per unit volume, $\mathbf{t}$ is the traction specified on $\Gamma_N^p$, $\boldsymbol{\epsilon}$ is the strain tensor, $\boldsymbol{\epsilon}^e$ and $\boldsymbol{\epsilon}^p$ are the elastic and plastic parts of strain tensor respctively, $\boldsymbol{\sigma}$ is the Cauchy stress tensor, $\mathbb{D}$ is the fourth order symmetric positive definite anisotropic elasticity tensor, $\boldsymbol{\alpha}$ is the Biot tensor and $\gamma\geq 0$ is the plastic multiplier satisfying the complementarity condition
\begin{equation}
\label{comple}
\left.\begin{array}{c}
\gamma \Phi =0\\
\gamma > 0 \leftrightarrow \Phi=0\\
\gamma=0 \leftrightarrow \Phi<0 
\end{array}\right.
\end{equation}
The inverse of the constitutive law is
\begin{align}
\label{invconstitutive}
&\boldsymbol{\epsilon}^e=\mathbb{D}^{-1}(\boldsymbol{\sigma}+\boldsymbol{\alpha}p)\equiv \mathbb{D}^{-1}\boldsymbol{\sigma}+\frac{C}{3}\mathbf{B}p
\end{align}
where $C(>0)$ is a generalized Hooke's law constant (see \cite{cheng}) and $\mathbf{B}$ is the Skempton pore pressure coefficient (see \cite{skempton-1954}). 
\subsection{Increment in fluid content}
Fluid content is defined as (see \cite{coussy})
\begin{align}
\label{fluidcontent}
\zeta=\frac{1}{M}p+\boldsymbol{\alpha}:\boldsymbol{\epsilon}^e+\phi^p \equiv Cp+\frac{1}{3}C\mathbf{B}:\boldsymbol{\sigma}+\phi^p
\end{align}
where $M(>0)$ is the Biot modulus (see \cite{biot3,cheng}) and $\phi^p$ is a plastic porosity 
given by (see \cite{coussy})
\begin{align}
\label{pporosity}
\phi^p \equiv \boldsymbol{\beta}:\boldsymbol{\epsilon}^p
\end{align}
where $\boldsymbol{\beta}$ is a material parameter.
\subsection{Variational statements and spaces used}
The problem statement is: find $p_h\in W_h$, $\mathbf{z}_h\in \mathbf{V}_h$ and $\mathbf{u}_h\in \mathbf{U}_h$ such that
\begin{align}
\nonumber
&C(p_h^{k,n+1}-p_h^n,\theta_h)_{\Omega}+\Delta t(\nabla \cdot \mathbf{z}_h^{k,n+1},\theta_h)_{\Omega}+(\phi^{p^{k,n+1}}-\phi^{p^n},\theta_h)_{\Omega}\\
\label{one}
&=\Delta t(q^{n+1},\theta_h)_{\Omega}-\frac{C}{3}(\mathbf{B}:(\boldsymbol{\sigma}^{k-1,n+1}-\boldsymbol{\sigma}^n),\theta_h)_{\Omega}\\
\label{two}
&(\boldsymbol{\kappa}^{-1}\mathbf{z}_h^{k,n+1},\mathbf{v}_h)_{\Omega}=-(g,\mathbf{v}_h\cdot \mathbf{n})_{\Gamma_D^f}+(p_h^{k,n+1},\nabla \cdot \mathbf{v}_h)_{\Omega}+(\rho_0 \mathbf{g},\mathbf{v}_h)_{\Omega}
\\
\label{three}
&(\mathbf{t}^{n+1},\mathbf{q}_h)_{\Gamma_N^p} - (\boldsymbol{\sigma}^{k,n+1}:\boldsymbol{\epsilon}^e(\mathbf{q}_h))_{\Omega} + (\mathbf{f}^{n+1},\mathbf{q}_h)_{\Omega}=0
\end{align}
where the superscript $k$ refers to the coupling iteration count and superscript $n$ refers to time step number and the mixed finite element space \mbox{$W_h \times \mathbf{V}_h$} and conforming Galerkin space \mbox{$\mathbf{U}_h$} are given by
\begin{align*}
&W_h= \big\{\theta_h:\theta_h\vert_E\in P_0(E)\,\,\forall E\in \mathscr{T}_h\big\}\\
&\mathbf{V}_h=\big\{\mathbf{v}_h:\mathbf{v}_h\vert_E\leftrightarrow \hat{\mathbf{v}}\vert_{\hat{E}}\in \hat{\mathbf{V}}(\hat{E})\,\,\forall E\in \mathscr{T}_h,\,\,\mathbf{v}_h \cdot \mathbf{n}=0\,\,\mathrm{on}\,\,\Gamma_N^f\big\}\\
&\mathbf{U}_h=\big\{\mathbf{q}_h:\mathbf{q}_h\vert_E\in Q_1(E)\,\,\forall E\in \mathscr{T}_h,\,\,\mathbf{q}_h=\mathbf{0}\,\,\mathrm{on}\,\,\Gamma_D^p\big\}
\end{align*}   
where $P_0$ represents the space of constants, $Q_1$ represents the space of trilinears and the details of $\hat{\mathbf{V}}(\hat{E})$ are given in \cite{danacg}.
The details of \eqref{one} and \eqref{three} are given in Appendix \ref{discretestate}.
\section{Convergence criterion}
To arrive at the convergence criterion, we first arrive at the statement of convergence. We then proceed to simply the convergence criterion by expressing it in terms of computed quantities over every iteration.
\subsection{Statement of convergence}
\begin{figure}[htb!]
\centering
\includegraphics[scale=0.4]{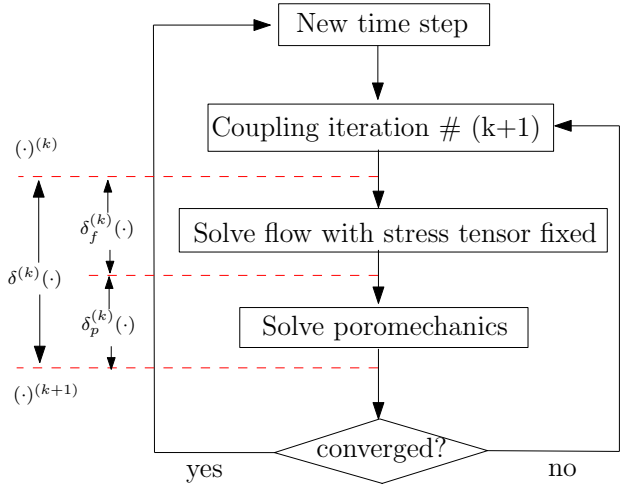}
\caption{The contraction mapping is in terms of quantities $\delta(\cdot)^{(k)}$ and $\delta(\cdot)^{(k)}_p$.}
\label{algorithmep}
\end{figure}
We follow a procedure similar to that in \mbox{\cite{danacg}} to arrive at the statement of convergence. As elucidated in Figure \ref{algorithmep}, we use the notations $\delta^{(k)}_f (\cdot)$ and $\delta^{(k)}_p (\cdot)$ for the change in the quantity $(\cdot)$ during the flow and poromechanics solves respectively over the $(k+1)^{th}$ coupling iteration and $\delta^{(k)} (\cdot)$ for the change in the quantity $(\cdot)$ over the $(k+1)^{th}$ coupling iteration at any time level such that
\begin{align*}
\delta^{(k)} (\cdot)\equiv (\cdot)^{(k+1)}-(\cdot)^{k}=\delta^{(k)}_f (\cdot)+\delta^{(k)}_p (\cdot)
\end{align*}
With that in mind, \eqref{one}-\eqref{three} are written as
\begin{align}
\label{msone}
&C(\delta^{(k)}p_h,\theta_h)_{\Omega}+\Delta t(\nabla \cdot
  \delta^{(k)}\mathbf{z}_h,\theta_h)_{\Omega}+(\delta^{(k)}\phi^p,\theta_h)_{\Omega}=-\frac{C}{3}(\mathbf{B}:\delta^{(k-1)}\boldsymbol{\sigma},\theta_h)_{\Omega}\\
\label{mstwo}
&(\boldsymbol{\kappa}^{-1}\delta^{(k)} \mathbf{z}_h, \mathbf{v}_h)_{\Omega}=(\delta^{(k)} p_h,\nabla \cdot \mathbf{v}_h)_{\Omega}\\
\label{msthree}
&(\delta^{(k)} \boldsymbol{\sigma}:\boldsymbol{\epsilon}^e(\mathbf{q}_h))_{\Omega}=0
\end{align}
\begin{theorem}\label{map}
The fixed stress split scheme is a contraction map with contraction constant
  $\gamma=\frac{1}{2}$ given by 
\begin{align}
\nonumber
&\Vert
  \mathbf{B}:\delta^{(k)}\boldsymbol{\sigma}\Vert_{\Omega}^2+\overbrace{\Vert\delta^{(k)}p_h\Vert_{\Omega}^2}^{\geq
  0}+\overbrace{\frac{3}{C}\Delta t
\Vert\boldsymbol{\kappa}^{-1/2}\delta^{(k)} \mathbf{z}_h
  \Vert^2_{\Omega}}^{\geq 0}+\overbrace{
  \frac{3}{C}(\delta^{(k)} \boldsymbol{\sigma}:\mathbb{D}^{-1}\delta^{(k)}\boldsymbol{\sigma})_{\Omega}}^{\geq 0}+\overbrace{\frac{3}{2C^2}\Vert \delta^{(k)}\zeta\Vert_{\Omega}^2}^{\geq 0}\\
\nonumber   
  &-\overbrace{\frac{3}{C^2}\bigg[\frac{7}{2}\Vert
  \delta^{(k)}_p\zeta-\delta^{(k)}_p\phi^p\Vert_{\Omega}^2+\frac{1}{2}\Vert
  \delta^{(k)}\phi^p\Vert_{\Omega}^2+\frac{C}{3}(\mathbf{B}:\delta^{(k)}\boldsymbol{\sigma},\delta^{(k)}\phi^p)_{\Omega}\bigg]}^{\mathrm{\rightarrow
  0}}\leq \gamma\Vert \mathbf{B}:\delta^{(k-1)}\boldsymbol{\sigma}\Vert_{\Omega}^2
\end{align}
  where the term
  \begin{align*}
  \frac{3}{C^2}\bigg[\frac{7}{2}\Vert
  \delta^{(k)}_p\zeta-\delta^{(k)}_p\phi^p\Vert_{\Omega}^2+\frac{1}{2}\Vert
  \delta^{(k)}\phi^p\Vert_{\Omega}^2+\frac{C}{3}(\mathbf{B}:\delta^{(k)}\boldsymbol{\sigma},\delta^{(k)}\phi^p)_{\Omega}\bigg]\end{align*}
  is driven to a small value by the convergence criterion. The reason behind driving this term to a small value is to minimize the influence of the negativity of the term on the quality of the contraction mapping. 
\end{theorem}
\begin{proof}
$\bullet$ \textbf{Step 1: Flow equations}\\ 
Testing \eqref{msone} with $\theta_h\equiv \delta^{(k)} p_h$ and \eqref{mstwo} with $\mathbf{v}_h\equiv \delta^{(k)} \mathbf{z}_h$, we get
\begin{align}
\label{msfour}
&C\Vert\delta^{(k)}p_h\Vert_{\Omega}^2+\Delta t(\nabla \cdot \delta^{(k)}\mathbf{z}_h,\delta^{(k)} p_h)_{\Omega}+(\delta^{(k)}\phi^p,\delta^{(k)} p_h)_{\Omega}=-\frac{C}{3}(\mathbf{B}:\delta^{(k-1)}\boldsymbol{\sigma},\delta^{(k)} p_h)_{\Omega}\\
\label{msfive}
&\Vert\boldsymbol{\kappa}^{-1/2}\delta^{(k)} \mathbf{z}_h \Vert^2_{\Omega}=(\delta^{(k)} p_h,\nabla \cdot \delta^{(k)} \mathbf{z}_h)_{\Omega}
\end{align}
From \eqref{msfour} and \eqref{msfive}, we get
\begin{align}
\label{mssix}
&C\Vert\delta^{(k)}p_h\Vert_{\Omega}^2+\Delta t
\Vert\boldsymbol{\kappa}^{-1/2}\delta^{(k)} \mathbf{z}_h \Vert^2_{\Omega}+(\delta^{(k)}\phi^p,\delta^{(k)} p_h)_{\Omega}=-\frac{C}{3}(\mathbf{B}:\delta^{(k-1)}\boldsymbol{\sigma},\delta^{(k)} p_h)_{\Omega}
\end{align}
$\bullet$ \textbf{Step 2: Poromechanics equations}\\
Testing \eqref{msthree} with $\mathbf{q}_h\equiv \delta^{(k)}\mathbf{u}_h$, we get
\begin{align}
\label{eq1}
&(\delta^{(k)} \boldsymbol{\sigma}:\delta^{(k)}\boldsymbol{\epsilon}^e)_{\Omega}=0
\end{align}
We now invoke \eqref{invconstitutive} to arrive at $\delta^{(k)}\boldsymbol{\epsilon}^e=\mathbb{D}^{-1}\delta^{(k)}\boldsymbol{\sigma}+\frac{C}{3}\mathbf{B}\delta^{(k)}p_h$. Substituting in \eqref{eq1}, we get
\begin{align}
\label{msseven}
&(\delta^{(k)} \boldsymbol{\sigma}:\mathbb{D}^{-1}\delta^{(k)}\boldsymbol{\sigma})_{\Omega}+\frac{C}{3}(\mathbf{B}:\delta^{(k)} \boldsymbol{\sigma},\delta^{(k)}p_h)_{\Omega}=0
\end{align}
$\bullet$ \textbf{Step 3: Combining flow and poromechanics equations}\\
Adding \eqref{mssix} and \eqref{msseven}, we get
\begin{align}
\nonumber
&C\Vert\delta^{(k)}p_h\Vert_{\Omega}^2+\Delta t
\Vert\boldsymbol{\kappa}^{-1/2}\delta^{(k)} \mathbf{z}_h \Vert^2_{\Omega}+(\delta^{(k)}\phi^p,\delta^{(k)} p_h)_{\Omega}+(\delta^{(k)}
  \boldsymbol{\sigma}:\mathbb{D}^{-1}\delta^{(k)}\boldsymbol{\sigma})_{\Omega}+\frac{C}{3}(\mathbf{B}:\delta^{(k)}
  \boldsymbol{\sigma},\delta^{(k)}p_h)_{\Omega}\\
  \label{use1}
  &=-\frac{C}{3}(\mathbf{B}:\delta^{(k-1)}\boldsymbol{\sigma},\delta^{(k)} p_h)_{\Omega}
\end{align}
$\bullet$ \textbf{Step 4: Variation in fluid content}\\
In lieu of \eqref{fluidcontent}, we write
\begin{align}
\label{fluidcontentfull}
&\delta^{(k)}\zeta=C\delta^{(k)}p_h+\frac{C}{3}\mathbf{B}:\delta^{(k)}\boldsymbol{\sigma}+\delta^{(k)}\phi^p
\end{align}
which can be written as
\begin{align*}
  \Vert \delta^{(k)}\zeta \Vert_{\Omega}^2&=C^2\Vert
  \delta^{(k)}p_h\Vert_{\Omega}^2+\frac{C^2}{9}\Vert
  \mathbf{B}:\delta^{(k)}\boldsymbol{\sigma}\Vert_{\Omega}^2+\Vert
  \delta^{(k)}\phi^p\Vert_{\Omega}^2+\frac{2C^2}{3}(\delta^{(k)}p_h,\mathbf{B}:\delta^{(k)}\boldsymbol{\sigma})_{\Omega}\\
  &+\frac{2C}{3}(\mathbf{B}:\delta^{(k)}\boldsymbol{\sigma},\delta^{(k)}\phi^p)_{\Omega}+2C(\delta^{(k)}p_h,\delta^{(k)}\phi^p)_{\Omega}
\end{align*}
Dividing throughout by $2C$, we get
\begin{align}
\nonumber
&\frac{1}{2C}\Vert \delta^{(k)}\zeta\Vert_{\Omega}^2-\frac{C}{2}\Vert\delta^{(k)}p_h\Vert_{\Omega}^2-\frac{C}{18}\Vert \mathbf{B}:\delta^{(k)}\boldsymbol{\sigma}\Vert_{\Omega}^2-\frac{1}{2C}\Vert \delta^{(k)}\phi^p\Vert_{\Omega}^2\\
\label{daal}
&-\frac{1}{3}(\mathbf{B}:\delta^{(k)}\boldsymbol{\sigma},\delta^{(k)}\phi^p)_{\Omega}=\frac{C}{3}(\mathbf{B}:\delta^{(k)} \boldsymbol{\sigma},\delta^{(k)}p_h)_{\Omega}+(\delta^{(k)}\phi^p,\delta^{(k)} p_h)_{\Omega}
\end{align}
From \eqref{use1} and \eqref{daal}, we get
\begin{align}
\nonumber
&C\Vert\delta^{(k)}p_h\Vert_{\Omega}^2+\Delta t
\Vert\boldsymbol{\kappa}^{-1/2}\delta^{(k)} \mathbf{z}_h \Vert^2_{\Omega}+(\delta^{(k)} \boldsymbol{\sigma}:\mathbb{D}^{-1}\delta^{(k)}\boldsymbol{\sigma})_{\Omega}+\frac{1}{2C}\Vert \delta^{(k)}\zeta\Vert_{\Omega}^2\\
\nonumber
&-\frac{C}{2}\Vert\delta^{(k)}p_h\Vert_{\Omega}^2-\frac{C}{18}\Vert \mathbf{B}:\delta^{(k)}\boldsymbol{\sigma}\Vert_{\Omega}^2-\frac{1}{2C}\Vert \delta^{(k)}\phi^p\Vert_{\Omega}^2-\frac{1}{3}(\mathbf{B}:\delta^{(k)}\boldsymbol{\sigma},\delta^{(k)}\phi^p)_{\Omega}\\
\label{use2}
&=-\frac{C}{3}(\mathbf{B}:\delta^{(k-1)}\boldsymbol{\sigma},\delta^{(k)} p_h)_{\Omega}
\end{align}
Adding and subtracting $\frac{C}{3}\Vert \mathbf{B}:\delta^{(k)}\boldsymbol{\sigma}\Vert_{\Omega}^2$ to the LHS of \eqref{use2} results in
\begin{align}
\nonumber
&\frac{C}{3}\Vert \mathbf{B}:\delta^{(k)}\boldsymbol{\sigma}\Vert_{\Omega}^2+\frac{C}{2}\Vert\delta^{(k)}p_h\Vert_{\Omega}^2+\Delta t
\Vert\boldsymbol{\kappa}^{-1/2}\delta^{(k)} \mathbf{z}_h \Vert^2_{\Omega}+(\delta^{(k)} \boldsymbol{\sigma}:\mathbb{D}^{-1}\delta^{(k)}\boldsymbol{\sigma})_{\Omega}\\
\nonumber
&+\frac{1}{2C}\Vert \delta^{(k)}\zeta\Vert_{\Omega}^2-\frac{7C}{18}\Vert \mathbf{B}:\delta^{(k)}\boldsymbol{\sigma}\Vert_{\Omega}^2-\frac{1}{2C}\Vert \delta^{(k)}\phi^p\Vert_{\Omega}^2-\frac{1}{3}(\mathbf{B}:\delta^{(k)}\boldsymbol{\sigma},\delta^{(k)}\phi^p)_{\Omega}\\
\nonumber
&=-\frac{C}{3}(\mathbf{B}:\delta^{(k-1)}\boldsymbol{\sigma},\delta^{(k)} p_h)_{\Omega}
\end{align}
Multiplying throughout by $C$ results in
\begin{align}
\nonumber
&\frac{C^2}{3}\Vert \mathbf{B}:\delta^{(k)}\boldsymbol{\sigma}\Vert_{\Omega}^2+\frac{C^2}{2}\Vert\delta^{(k)}p_h\Vert_{\Omega}^2+C\Delta t
\Vert\boldsymbol{\kappa}^{-1/2}\delta^{(k)} \mathbf{z}_h \Vert^2_{\Omega}+C(\delta^{(k)} \boldsymbol{\sigma}:\mathbb{D}^{-1}\delta^{(k)}\boldsymbol{\sigma})_{\Omega}\\
\nonumber
&+\frac{1}{2}\Vert \delta^{(k)}\zeta\Vert_{\Omega}^2-\frac{7C^2}{18}\Vert \mathbf{B}:\delta^{(k)}\boldsymbol{\sigma}\Vert_{\Omega}^2-\frac{1}{2}\Vert \delta^{(k)}\phi^p\Vert_{\Omega}^2-\frac{C}{3}(\mathbf{B}:\delta^{(k)}\boldsymbol{\sigma},\delta^{(k)}\phi^p)_{\Omega}\\
\label{use3}
&=-\frac{C^2}{3}(\mathbf{B}:\delta^{(k-1)}\boldsymbol{\sigma},\delta^{(k)} p_h)_{\Omega}
\end{align}
$\bullet$ \textbf{Step 5: Invoking the fixed stress constraint}\\
Using \eqref{fluidcontent} and the fixed stress constraint during the flow solve, we get
\begin{align*}
&\delta^{(k)}_f\zeta=C\delta_f^{(k)}p_h+\frac{C}{3}\mathbf{B}:\cancelto{\mathbf{0}}{\delta^{(k)}_f\boldsymbol{\sigma}}+\delta_f^{(k)}\phi^p
\end{align*}
Further, since the pore pressure is frozen during the poromechanical solve, we have $\delta^{(k)}_fp_h=\delta^{(k)}p_h$. As a result, we can write
\begin{align}
\label{fluidcontentflow}
&\delta^{(k)}_f\zeta=C\delta^{(k)}p_h+\delta_f^{(k)}\phi^p
\end{align}
Subtracting \eqref{fluidcontentflow} from \eqref{fluidcontentfull}, we can write
\begin{align}
\label{chahiye}
&\delta^{(k)}_p\zeta=\frac{C}{3}\mathbf{B}:\delta^{(k)}\boldsymbol{\sigma}+\delta^{(k)}_p\phi^p
\end{align}
which implies that
\begin{align}
\label{convergence}
&\Vert \delta^{(k)}_p\zeta-\delta^{(k)}_p\phi^p\Vert_{\Omega}^2=\frac{C^2}{9}\Vert \mathbf{B}:\delta^{(k)}\boldsymbol{\sigma}\Vert_{\Omega}^2
\end{align}
Noting \eqref{convergence}, we can write \eqref{use3} as
\begin{align}
\nonumber
&\frac{C^2}{3}\Vert \mathbf{B}:\delta^{(k)}\boldsymbol{\sigma}\Vert_{\Omega}^2+\frac{C^2}{2}\Vert\delta^{(k)}p_h\Vert_{\Omega}^2+C\Delta t\Vert\boldsymbol{\kappa}^{-1/2}\delta^{(k)} \mathbf{z}_h \Vert^2_{\Omega}
+C(\delta^{(k)} \boldsymbol{\sigma}:\mathbb{D}^{-1}\delta^{(k)}\boldsymbol{\sigma})_{\Omega}\\
\nonumber
  +&\bigg[\frac{1}{2}\Vert \delta^{(k)}\zeta\Vert_{\Omega}^2-\frac{7}{2}\Vert \delta^{(k)}_p\zeta-\delta^{(k)}_p\phi^p\Vert_{\Omega}^2-\frac{1}{2}\Vert \delta^{(k)}\phi^p\Vert_{\Omega}^2-\frac{C}{3}(\mathbf{B}:\delta^{(k)}\boldsymbol{\sigma},\delta^{(k)}\phi^p)_{\Omega}\bigg]\\
\label{mseight}
&=-\frac{C^2}{3}(\mathbf{B}:\delta^{(k-1)}\boldsymbol{\sigma},\delta^{(k)} p_h)_{\Omega}
\end{align}
$\bullet$ \textbf{Step 6: Invoking the Young's inequality}\\
We invoke the Young's inequality (see \cite{steele})
for the RHS of \eqref{mseight} as follows
\begin{align}
\label{young}
&-\frac{C^2}{3}(\mathbf{B}:\delta^{(k-1)}\boldsymbol{\sigma},\delta^{(k)} p_h)_{\Omega}\leq \frac{C^2}{3}\bigg[\frac{1}{2} \Vert \mathbf{B}:\delta^{(k-1)}\boldsymbol{\sigma}\Vert_{\Omega}^2+\frac{1}{2}\Vert\delta^{(k)} p_h \Vert_{\Omega}^2\bigg]
\end{align}
In lieu of \eqref{young}, we write \eqref{mseight} as
\begin{align}
\nonumber
&\frac{C^2}{3}\Vert \mathbf{B}:\delta^{(k)}\boldsymbol{\sigma}\Vert_{\Omega}^2+\frac{C^2}{3}\Vert\delta^{(k)}p_h\Vert_{\Omega}^2+C\Delta t
\Vert\boldsymbol{\kappa}^{-1/2}\delta^{(k)} \mathbf{z}_h \Vert^2_{\Omega}+C(\delta^{(k)} \boldsymbol{\sigma}:\mathbb{D}^{-1}\delta^{(k)}\boldsymbol{\sigma})_{\Omega}\\
\nonumber
  +&\bigg[\frac{1}{2}\Vert
  \delta^{(k)}\zeta\Vert_{\Omega}^2-\frac{7}{2}\Vert \delta^{(k)}_p\zeta-\delta^{(k)}_p\phi^p\Vert_{\Omega}^2-\frac{1}{2}\Vert \delta^{(k)}\phi^p\Vert_{\Omega}^2-\frac{C}{3}(\mathbf{B}:\delta^{(k)}\boldsymbol{\sigma},\delta^{(k)}\phi^p)_{\Omega}\bigg]\\
\nonumber
&\leq \frac{C^2}{6}\Vert \mathbf{B}:\delta^{(k-1)}\boldsymbol{\sigma}\Vert_{\Omega}^2
\end{align}
which can be written as
\begin{align}
\nonumber
&\Vert
  \mathbf{B}:\delta^{(k)}\boldsymbol{\sigma}\Vert_{\Omega}^2+\overbrace{\Vert\delta^{(k)}p_h\Vert_{\Omega}^2}^{\geq
  0}+\overbrace{\frac{3}{C}\Delta t
\Vert\boldsymbol{\kappa}^{-1/2}\delta^{(k)} \mathbf{z}_h
  \Vert^2_{\Omega}}^{\geq 0}+\overbrace{
  \frac{3}{C}(\delta^{(k)} \boldsymbol{\sigma}:\mathbb{D}^{-1}\delta^{(k)}\boldsymbol{\sigma})_{\Omega}}^{\geq 0}+\overbrace{\frac{3}{2C^2}\Vert \delta^{(k)}\zeta\Vert_{\Omega}^2}^{\geq 0}\\
\nonumber   
  &-\overbrace{\frac{3}{C^2}\bigg[\frac{7}{2}\Vert
  \delta^{(k)}_p\zeta-\delta^{(k)}_p\phi^p\Vert_{\Omega}^2+\frac{1}{2}\Vert
  \delta^{(k)}\phi^p\Vert_{\Omega}^2+\frac{C}{3}(\mathbf{B}:\delta^{(k)}\boldsymbol{\sigma},\delta^{(k)}\phi^p)_{\Omega}\bigg]}^{\mathrm{\rightarrow
  0}}\leq \frac{1}{2}\Vert \mathbf{B}:\delta^{(k-1)}\boldsymbol{\sigma}\Vert_{\Omega}^2
\end{align}
\end{proof}
\subsection{Derivation of criterion}
We desire to drive the following quantity to zero
\begin{align}
\label{design}
\frac{3}{C^2}\bigg[\frac{7}{2}\Vert
  \delta^{(k)}_p\zeta-\delta^{(k)}_p\phi^p\Vert_{\Omega}^2+\frac{1}{2}\Vert
  \delta^{(k)}\phi^p\Vert_{\Omega}^2+\frac{C}{3}(\mathbf{B}:\delta^{(k)}\boldsymbol{\sigma},\delta^{(k)}\phi^p)_{\Omega}\bigg]\end{align}
Noting \eqref{chahiye}, we can write
\begin{align}
\label{abhi}
\frac{C}{3}(\mathbf{B}:\delta^{(k)}\boldsymbol{\sigma},\delta^{(k)}\phi^p)_{\Omega}&=(\delta_p^{(k)}\zeta-\delta^{(k)}_p\phi^p,\delta^{(k)}\phi^p)_{\Omega}
\end{align}
Using \eqref{abhi}, we can write \eqref{design} as
\begin{align*}
  \frac{3}{C^2}\bigg[\frac{7}{2}\Vert
  \delta^{(k)}_p\zeta-\delta^{(k)}_p\phi^p\Vert_{\Omega}^2+\frac{1}{2}\Vert
  \delta^{(k)}\phi^p\Vert_{\Omega}^2+(\delta_p^{(k)}\zeta-\delta^{(k)}_p\phi^p,\delta^{(k)}\phi^p)_{\Omega}\bigg]
\end{align*}
which can also be written as
\begin{align*}
  \frac{3}{C^2}\bigg[\frac{7}{2}\Vert
  \delta^{(k)}_p\zeta-\delta^{(k)}_p\phi^p\Vert_{\Omega}^2+\frac{1}{2}\Vert
  \delta^{(k)}_p\phi^p\Vert_{\Omega}^2+(\delta_p^{(k)}\zeta-\delta^{(k)}_p\phi^p,\delta^{(k)}_p\phi^p)_{\Omega}\bigg]
\end{align*}
which, given that $\delta^{(k)}_f\phi^p$ (see Appendix \ref{understand}), can also be written as
\begin{align*}
  \frac{3}{2C^2}\bigg[6\Vert \delta^{(k)}_p\zeta-\delta^{(k)}_p\phi^p\Vert_{\Omega}^2+\Vert\delta^{(k)}_p\zeta \Vert_{\Omega}^2\bigg]
\end{align*}
which, in lieu of \eqref{invconstitutive} and \eqref{pporosity}, can also be written as
\begin{align*}
  \frac{3}{2C^2}\bigg[6\Vert \boldsymbol{\alpha}:\delta^{(k)}_p
  \boldsymbol{\epsilon}^e\Vert_{\Omega}^2+\Vert \boldsymbol{\alpha}:\delta^{(k)}_p
  \boldsymbol{\epsilon}^e+\boldsymbol{\beta}:\delta^{(k)}_p
  \boldsymbol{\epsilon}^p \Vert_{\Omega}^2\bigg]
\end{align*}
As a result, we pose the convergence criterion as
\begin{align}
\label{ccriterion}
  &6\Vert \boldsymbol{\alpha}:\delta^{(k)}_p
  \boldsymbol{\epsilon}^e\Vert_{\Omega}^2+\Vert \boldsymbol{\alpha}:\delta^{(k)}_p
  \boldsymbol{\epsilon}^e+\boldsymbol{\beta}:\delta^{(k)}_p
  \boldsymbol{\epsilon}^p \Vert_{\Omega}^2\leq TOL
\end{align}
where $TOL>0$ is a pre-specified tolerance and represents a small value.
\section{Results and conclusions}
We derived the convergence criterion for a staggered solution algorithm for anisotropic poroelastoplasticity coupled with single phase flow. We applied the principle of contraction mapping which states that the difference of iterates monotonically reduces with coupling iterations for convergence to be achieved at each time step. The convergence criterion is designed such that the statement of contraction mapping holds true.
\appendix
\section{Discrete variational statements}\label{discretestate}
In lieu of \eqref{fluidcontent}, we write mass conservation equation as
\begin{align}
\nonumber
&\frac{\partial}{\partial t}(Cp+\frac{C}{3}\mathbf{B}:\boldsymbol{\sigma}+\phi^p)+\nabla \cdot \mathbf{z}=q\\
\label{ek}
&C\frac{\partial p}{\partial t} +\nabla \cdot \mathbf{z}=q-\frac{C}{3}\mathbf{B}:\frac{\partial \boldsymbol{\sigma}}{\partial t}-\frac{\partial \phi^p}{\partial t}
\end{align}
The discrete in time form of \eqref{ek} in the $(n+1)^{th}$ time step is written as
\begin{align*}
&C\frac{1}{\Delta t}(p^{k,n+1}-p^n) +\nabla \cdot \mathbf{z}^{k,n+1}=q^{n+1}-\frac{1}{\Delta t}\frac{C}{3}\mathbf{B}:(\boldsymbol{\sigma}^{k,n+1}-\boldsymbol{\sigma}^n)-\frac{1}{\Delta t}(\phi^{p^{k,n+1}}-\phi^{p^n})
\end{align*}
where $\Delta t$ is the time step. The fixed stress split constraint implies that $\boldsymbol{\sigma}^{k,n+1}$ gets replaced by $\boldsymbol{\sigma}^{k-1,n+1}$ as $\boldsymbol{\sigma}$ is fixed during the flow solve. The modified equation is written as
\begin{align}
\nonumber
&C(p^{k,n+1}-p^n)+\Delta t\nabla \cdot \mathbf{z}^{k,n+1}=\Delta t q^{n+1}-\frac{C}{3}\mathbf{B}:(\boldsymbol{\sigma}^{k-1,n+1}-\boldsymbol{\sigma}^n)-(\phi^{p^{k,n+1}}-\phi^{p^n})
\end{align}
As a result, the discrete weak form of mass conservation is given by
\begin{align}
\nonumber
&C(p_h^{k,n+1}-p_h^n,\theta_h)_{\Omega}+\Delta t(\nabla \cdot \mathbf{z}_h^{k,n+1},\theta_h)_{\Omega}+(\phi^{p^{k,n+1}}-\phi^{p^n},\theta_h)_{\Omega}\\
\nonumber
&=\Delta t(q^{n+1},\theta_h)_{\Omega}-\frac{C}{3}(\mathbf{B}:(\boldsymbol{\sigma}^{k-1,n+1}-\boldsymbol{\sigma}^n),\theta_h)_{\Omega}
\end{align}
The weak form of the linear momentum balance is given by
\begin{align}
\label{app1}
(\nabla \cdot \boldsymbol{\sigma},\mathbf{q})_{\Omega}+(\mathbf{f}\cdot \mathbf{q})_{\Omega}=0\qquad (\forall\,\,\mathbf{q}\in \mathbf{U}(\Omega))
\end{align}
where $\mathbf{U}(\Omega)$ is given by
\begin{align*}
\mathbf{U}(\Omega)\equiv \big\{\mathbf{q}=(u,v,w):u,v,w\in H^1(\Omega),\mathbf{q}=\mathbf{0}\,\,\mathrm{on}\,\,\Gamma_D^p\big\}
\end{align*}
where $H^m(\Omega)$ is defined, in general, for any integer $m\geq 0$ as
\begin{align*}
H^m(\Omega)\equiv\big\{w:D^{\alpha}w\in L^2(\Omega)\,\,\forall |\alpha| \leq m \big\},
\end{align*}
where the derivatives are taken in the sense of distributions and given by
\begin{align*}
D^{\alpha}w=\frac{\partial^{|\alpha|}w}{\partial x_1^{\alpha_1}..\partial x_n^{\alpha_n}},\,\,|\alpha|=\alpha_1+\cdots+\alpha_n,
\end{align*} 
We know from tensor calculus that
\begin{align}
\label{app2}
(\nabla \cdot \boldsymbol{\sigma},\mathbf{q})_{\Omega}\equiv (\nabla ,\boldsymbol{\sigma}\mathbf{q})_{\Omega}-(\boldsymbol{\sigma}:\nabla \mathbf{q})_{\Omega}
\end{align}
Further, using the divergence theorem and the symmetry of $\boldsymbol{\sigma}$, we arrive at
\begin{align}
\label{app3}
(\nabla ,\boldsymbol{\sigma}\mathbf{q})_{\Omega}\equiv (\mathbf{q},\boldsymbol{\sigma}\mathbf{n})_{\partial \Omega}
\end{align}
We decompose $\nabla \mathbf{q}$ into a symmetric part $(\nabla \mathbf{q})_{s}\equiv \frac{1}{2}\big(\nabla \mathbf{q}+(\nabla \mathbf{q})^T\big)\equiv \boldsymbol{\epsilon}^e(\mathbf{q})$ and skew-symmetric part $(\nabla \mathbf{q})_{ss}$ and note that the contraction between a symmetric and skew-symmetric tensor is zero to obtain
\begin{align}
\label{app4}
\boldsymbol{\sigma}:\nabla \mathbf{q}\equiv \boldsymbol{\sigma}:(\nabla \mathbf{q})_{s}+\cancelto{0}{\boldsymbol{\sigma}:(\nabla \mathbf{q})_{ss}}=\boldsymbol{\sigma}:\boldsymbol{\epsilon}^e(\mathbf{q})
\end{align}
From \eqref{app1}, \eqref{app2}, \eqref{app3} and \eqref{app4}, we get
\begin{align*}
&(\boldsymbol{\sigma}\mathbf{n},\mathbf{q})_{\partial \Omega} - (\boldsymbol{\sigma}:\boldsymbol{\epsilon}^e(\mathbf{q}))_{\Omega} + (\mathbf{f},\mathbf{q})_{\Omega}=0
\end{align*}
which, after invoking the traction boundary condition, results in the discrete weak form
\begin{align}
\nonumber
&(\mathbf{t}^{n+1},\mathbf{q}_h)_{\Gamma_N^p} - (\boldsymbol{\sigma}^{k,n+1}:\boldsymbol{\epsilon}^e(\mathbf{q}_h))_{\Omega} + (\mathbf{f}^{n+1},\mathbf{q}_h)_{\Omega}=0
\end{align}
\section{Change in plastic porosity over flow solve is zero}\label{understand}
\par To understand why $\delta_f^{(k)}\boldsymbol{\epsilon}^p=0$ and hence
$\delta_f^{(k)}\phi^p=0$, we present the basic algorithmic framework for the solution of elastoplastic equations: The system of equations \eqref{poromecheq} is first solved with $\gamma=0$ for a trial stress state $\boldsymbol{\sigma}^{\mathrm{trial}}$. 
\begin{itemize}
\item If $\Phi(\boldsymbol{\sigma}^{\mathrm{trial}})\leq 0$, then we proceed with $\boldsymbol{\sigma}=\boldsymbol{\sigma}^{\mathrm{trial}}$.
\item If $\Phi(\boldsymbol{\sigma}^{\mathrm{trial}})>0$, then the system \eqref{poromecheq} is solved with $\gamma>0$ thereby lending us plastic strain. The procedure to solve \eqref{poromecheq} with $\gamma>0$ is referred to as a return mapping algorithm~\cite{settari,liu,simo,ali,salari,borja,simo1,simo2,neto}. The solution $\boldsymbol{\sigma}^{\mathrm{return\,\,map}}$ of the return mapping algorithm is such that $\Phi(\boldsymbol{\sigma}^{\mathrm{return\,\,map}})=0$ and we proceed with $\boldsymbol{\sigma}=\boldsymbol{\sigma}^{\mathrm{return\,\,map}}$.
\end{itemize}
In summary, the solution $\boldsymbol{\sigma}$ is such that $\Phi(\boldsymbol{\sigma})\leq 0$. During the subsequent flow solve, since the stress tensor is fixed, the value of the yield function does not change i.e. $\Phi(\boldsymbol{\sigma})\leq 0$ still. This implies that $\gamma=0$ during the flow solve and the porous solid does not accumulate any plastic strain during the flow solve.

\vskip2mm

\bibliographystyle{unsrt}    
\bibliography{diss}

\end{document}